\documentclass[10pt]{article}

\usepackage[a4paper,text={16cm,23cm}]{geometry}
\usepackage{tikz}\usetikzlibrary{arrows,calc,decorations.text,decorations.pathreplacing}
\usepackage{qtree}
\usepackage{times}
\usepackage{color}
\usepackage[T1]{fontenc}
\usepackage[cp1252]{inputenc}
\usepackage[english]{babel}
\usepackage{array}
\usepackage{amsmath}
\usepackage{amssymb}
\usepackage{bbold}
\usepackage{graphicx}
\usepackage{mathrsfs}
\usepackage{enumerate}
\usepackage{hyperref}
\hypersetup{colorlinks=true,breaklinks=true,urlcolor=blue,linkcolor=blue,bookmarksopen=true}

\usepackage{amsthm}
\newcounter{count}[section]\numberwithin{count}{section}
\newtheorem{Thm}[count]{Theorem}
\newtheorem{Def}[count]{Definition}
\newtheorem{Prop}[count]{Proposition}
\newtheorem{Coro}[count]{Corollary}
\newtheorem{Lemm}[count]{Lemma}
\numberwithin{equation}{section}

\makeatletter
\renewenvironment{proof}[1][\proofname]{\par
  \pushQED{\qed}%
  \normalfont \topsep6\p@\@plus6\p@\relax
  \trivlist
  \item[\hskip\labelsep
        \bfseries
    #1\@addpunct{\scantokens{:}}]\ignorespaces
}{%
  \popQED\endtrivlist\@endpefalse
}
\makeatother

\newenvironment{remark}{\noindent\refstepcounter{count}\textbf{Remark \arabic{section}.\arabic{count}.}}{}
\newenvironment{acknowledgements}{\noindent\textbf{Acknowledgements:}}{}

\newcommand*{\R}{\mathbb{R}}

\newcommand*{\e}{\text{e}}

\newcommand*{\E}{\mathbb{E}}
\newcommand*{\prob}{\mathbb{P}}

\newcommand*{\indic}{\mathbb{1}}
\newcommand*{\Wass}{\mathcal{W}}
\newcommand*{\eqL}{\overset{\mathscr L}{=}}
\newcommand*{\leqL}{\overset{\mathscr L}{\leq}}

\begin{document}
\title{Quantitative speeds of convergence for exposure to food contaminants}
\author{
	Florian \textsc{Bouguet}\footnote{\texttt{Florian.Bouguet@univ-rennes1.fr}, UMR 6625 CNRS Institut de Recherche Math\'ematique de Rennes (IRMAR), Universit\'e de Rennes 1, Campus de Beaulieu, F-35042 \textsc{Rennes Cedex, France}.}
}
\date{June 11, 2014}
\maketitle

\begin{abstract}
In this paper we consider a class of piecewise-deterministic Markov processes modeling the quantity of a given food contaminant in the body. On the one hand, the amount of contaminant increases with random food intakes and, on the other hand, decreases thanks to the release rate of the body. Our aim is to provide quantitative speeds of convergence to equilibrium for the total variation and Wasserstein distances via coupling methods.\\


\end{abstract}


\section{Introduction}
\label{Sec_Introduction}
We study a piecewise-deterministic Markov process (PDMP) with pharmacokinetic properties; we refer to \cite{BCT08} and the references therein for details on the medical background motivating this model. This process is used to model the exposure to some chemical, such as methylmercury, which can be found in food. It has three random parts: the amount of contaminant ingested, the inter-intake times and the release rate of the body. Under some simple assumptions, with the help of Foster-Lyapounov methods, the geometric ergodicity has been proven in \cite{BCT08}; however, the rates of convergence are not explicit. The goal of our present paper is to provide quantitative exponential speeds of convergence to equilibrium for this PDMP, with the help of coupling methods. Note that another approach, quite recent, consists in using functional inequalities and hypocoercive methods (see \cite{Mon14a,Mon14b}) to quantify the ergodicity of non-reversible PDMPs.

Firstly, let us present the PDMP introduced in \cite{BCT08}, and recall its infinitesimal generator. We consider a test subject whose blood composition is constantly monitored. When he eats, a small amount of a given food contaminant (one may think of methylmercury for instance) is ingested; denote by $X_t$ the quantity of the contaminant in the body at time $t$. Between two contaminant intakes, the body purges itself so that the process $X$ follows the ordinary differential equation
\[\partial_tX_t=-\Theta X_t,\]
where $\Theta>0$ is a random metabolic parameter regulating the elimination speed. Following \cite{BCT08}, we will assume that $\Theta$ is constant between two food ingestions, which makes the trajectories of $X$ deterministic between two intakes. We also assume that the rate of intake depends only on the elapsed time since the last intake (which is realistic for a food contaminant present in a large variety of meals). As a matter of fact, \cite{BCT08} firstly deals with a slightly more general case, where $\partial_tX_t=-r(X_t,\Theta)$ and $r$ is a positive function. Our approach is likely to be easily generalizable if $r$ satisfies a condition like
\[r(x,\theta)-r(\tilde x,\theta)\geq C\theta(x-\tilde x),\]
but in the present paper we focus on the case $r(x,\theta)=\theta x$.

Define $T_0=0$ and $T_n$ the instant of $n^{\text{th}}$ intake. The random variables $\Delta T_n=T_n-T_{n-1}$, for $n\geq2$, are assumed to be i.i.d. and a.s. finite with distribution $G$. Let $\zeta$ be the hazard rate (or failure rate, see \cite{Lin86} or \cite{Bon95} for some reminders about reliability) of $G$; which means that $G([0,x])=1-\exp\left(-\int_0^x{\zeta(u)du}\right)$ by definition. In fact, there is no reason for $\Delta T_1=T_1$ to be distributed according to $G$, if the test subject has not eaten for a while before the beginning of the experience.  Let $N_t=\sum_{n=1}^\infty{\indic_{\{T_n\leq t\}}}$ be the total number of intakes at time $t$. For $n\geq1$, let
\[U_n=X_{T_n}-X_{T_n^-}\]
be the contaminant quantity taken at time $T_n$ (since $X$ is a.s. càdlàg, see a typical trajectory in Figure~\ref{Fig_TrajX}). Let $\Theta_n$ be the metabolic parameter between $T_{n-1}$ and $T_n$. We assume that the random variables $\{\Delta T_n,U_n,\Theta_n\}_{n\geq1}$ are independent. Finally, we denote by $F$ and $H$ the respective distributions of $U_1$ and $\Theta_1$. For obvious reasons, we assume also that the expectations of $F$ and $H$ are finite and $H((-\infty,0])=0$.

\begin{figure}[!ht]
\centering
\begin{tikzpicture}[scale=2.5]
\draw[->] (-0.2,0) -- (4.2,0);
\draw[->] (0,-0.2) -- (0,2.7);
\draw [domain=0:1,smooth,color=blue] plot (\x,{2*exp(2*(-\x))});
\draw [color=blue] (0.5,0.95) node{$\Theta_1$};
\draw [color=blue,dashed] (1,0.271) -- node[midway,left]{$U_1$} (1,1.5);
\draw [domain=1:3.5,smooth,color=blue] plot (\x,{1.5*exp(1*(1-\x))});
\draw [color=blue] (2.25,0.55) node{$\Theta_2$};
\draw [color=blue,dashed] (3.5,0.123) -- node[midway,left]{$U_2$} (3.5,2.5);
\draw [domain=3.5:4.2,smooth,color=blue] plot (\x,{2.5*exp(3*(3.5-\x))});
\draw [color=blue] (3.85,1.3) node{$\Theta_3$};
\draw (0,2) node[left]{$X_0$};
\draw[-|] (0,0)node[anchor=north east]{0} -- node[midway,below]{$\Delta T_1$} (1,0)node[anchor=north]{$T_1$};
\draw[-|] (1,0) -- node[midway,below]{$\Delta T_2$} (3.5,0)node[anchor=north]{$T_2$};
\end{tikzpicture}
\caption{Typical trajectory of $X$.}
\label{Fig_TrajX}\end{figure}
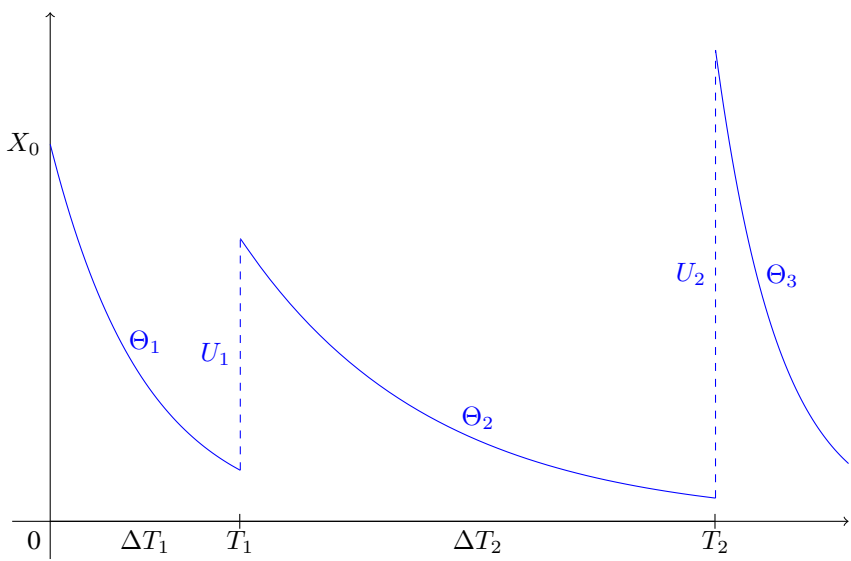

From now on, we make the following assumptions (only one assumption among \eqref{Eq_H4a} and \eqref{Eq_H4b} is required to be fullfiled):
\begin{align}
&F\text{ admits }f\text{ for density w.r.t. Lebesgue measure.}\tag{H1}\label{Eq_H1}\\
&G\text{ admits }g\text{ for density w.r.t. Lebesgue measure.}\tag{H2}\label{Eq_H2}\\
&\zeta\text{ is non-decreasing and }\lim_{t\to+\infty}\zeta(t)>0.\tag{H3}\label{Eq_H3}\\
&\eta\text{ is Hölder on }[0,1]\text{, where }\eta(x)=\frac12\int_\R{|f(u)-f(u-x)|du}.\tag{H4a}\label{Eq_H4a}\\
&f\text{ is Hölder on }\R_+\text{ and there exists }p>2\text{ such that }\lim_{x\to+\infty}x^pf(x)=0.\tag{H4b}\label{Eq_H4b}
\end{align}
From a modeling point of view, \eqref{Eq_H3} is reasonnable, since $\zeta$ models the hunger of the patient. Assumptions~\eqref{Eq_H4a} and \eqref{Eq_H4b} are purely technical, but reasonably mild.

Note that the process $X$ itself is not Markovian, since the jump rates depends on the time elapsed since the last intake. In order to deal with a PDMP, we consider the process $(X,\Theta,A)$, where
\[\Theta_t= \Theta_{N_t+1},\qquad A_t=t-T_{N_t}.\]
We call $\Theta$ the metabolic process, and $A$ the age process. The process $Y=(X,\Theta,A)$ is then a PDMP which possesses the strong Markov property (see \cite{Jac06}). Let $(P_t)_{t\geq0}$ be its semigroup; we denote by $\mu_0P_t$ the distribution of $Y_t$ when the law of $Y_0$ is $\mu_0$. Its infinitesimal generator is
\begin{equation}
\mathcal L\varphi(x,\theta,a) =\partial_a\varphi(x,\theta,a)-\theta x\partial_x\varphi(x,\theta,a)+\zeta(a)\int_{0}^\infty\int_0^\infty\big[\varphi(x+u,\theta',0)-\varphi(x,\theta,a)\big]H(d\theta')F(du).\label{Eq_GenL}
\end{equation}
Of course, if $\zeta$ is constant, then $(X,\Theta)$ is a PDMP all by itself. Let us recall that $\zeta$ being constant is equivalent to $G$ being an exponential distribution. Such a model is not relevant in this context, nevertheless it provides explicit speeds of convergence, as it will be seen in Section~\ref{SSec_PartCase}.

Now, we are able to state the following theorem, which is the main result of our paper; its proof will be postponed to Section~\ref{SSec_DetermDivision}.

\begin{Thm}\label{mainThm}
Let $\mu_0,\tilde\mu_0$ be distributions on $\R_+^3$. Then, there exist positive constants $C_1,C_2,C_3,C_4,v_1,v_2,v_3,v_4$ (see Remark~\ref{Rk_vitMainThm} for details) such that, for all $0<\alpha<\beta<1$:
\begin{enumerate}[(i)]
\item For all $t>0$,
\begin{equation}
\|\mu_0P_t-\tilde\mu_0P_t\|_{TV}\leq1-\left(1-C_1\e^{-v_1\alpha t}\right) \left(1-C_2\e^{-v_2(\beta-\alpha)t}\right) \left(1-C_3\e^{-v_3(1-\beta)t}\right)\left(1-C_4\e^{-v_4(\beta-\alpha)t}\right).
\label{Eq_MainTV}\end{equation}
\item For all $t>0$,
\begin{equation}
\Wass_1(\mu_0P_t,\tilde\mu_0P_t)\leq C_1\e^{-v_1\alpha t}+C_2\e^{-v_2(1-\alpha)t}.
\label{Eq_MainWass}\end{equation}
\end{enumerate}
\end{Thm}

\begin{remark}\label{Rk_vitMainThm}
The constants $C_i$ are not always explicit, since they are strongly linked to the Laplace transforms of the distributions considered, which are not always easy to deal with; the reader can find the details in the proof. However, the parameters $v_i$ are explicit and are provided throughout this paper. The speed $v_1$ comes from Theorem~\ref{vitAge} and Remark~\ref{Rk_vitAge}, and $v_2$ is provided by Corollary~\ref{vitWass2}. The only requirement for $v_3$ is that $G$ admits an exponential moment of order $v_3$ (see Remark~\ref{Rk_transLaplG}), and $v_4$ comes from Lemma~\ref{majSupEta}.
\end{remark}

The rest of this paper is organized as follows: in Section~\ref{Sec_Speeds}, we presents some heuristics of our method, and we provide tools to get lower bounds for the convergence speed to equilibrium of the PDMP, considering three successive phases (the age coalescence in Section~\ref{SSec_AgeCoal}, the Wasserstein coupling in Section~\ref{SSec_WassCoupl} and the total variation coupling in Section~\ref{SSec_TVCoupl}). Afterwards, we will use those bounds in Section~\ref{SSec_DetermDivision} to prove Theorem~\ref{mainThm}. Finally, a  particular and convenient case is treated in Section~\ref{SSec_PartCase}. Indeed, if the inter-intake times have an exponential distribution, better speeds of convergence may be provided.

\section{Explicit speeds of convergence}
\label{Sec_Speeds}
In this section, we draw our inspiration from coupling methods provided in \cite{CMP10,BCGMZ13} (for the TCP window size process), and in \cite{Lin86,Lin92} (for renewal processes). Two other standard references for coupling methods are \cite{Res92,Asm03}. The sequel provides not only existence and uniqueness of an invariant probability measure for $(P_t)$ (by consequence of our result, but it could also be proved by Foster-Lyapounov methods, which may require some slightly different assumptions, see \cite{MT93} or \cite{Hai10} for example) but also explicit exponential speeds of convergence to equilibrium for the total variation distance. The task is similar for convergence in Wasserstein distances.

Let us now briefly recall the definitions of the distances we use (see \cite{Vil09} for details). Let $\mu,\tilde \mu$ be two probability measures on $\R^d$ (we denote by $\mathscr M(E)$ the set of probability measures on $E$). Then, we call coupling of $\mu$ and $\tilde\mu$ any probability measure on $\R^d\times\R^d$ whose marginals are $\mu$ and $\tilde\mu$, and we denote by $\Gamma(\mu,\tilde\mu)$ the set of all the couplings of $\mu$ and $\tilde\mu$. Let $p\in[1,+\infty)$; if we denote by $\mathscr L(X)$ the law of any random vector $X$, the Wasserstein distance between $\mu$ and $\tilde\mu$ is defined by
\begin{equation}
\Wass_p(\mu,\tilde\mu)=\inf_{\mathscr L(X,\tilde X)\in\Gamma(\mu,\tilde\mu)}{\E\left[\|X-\tilde X\|^p\right]^{\frac1p}}.
\label{Eq_DefWass}\end{equation}
Similarly, the total variation distance between $\mu,\tilde\mu\in\mathscr M(\R^d)$ is defined by
\begin{equation}
\|\mu-\tilde\mu\|_{TV}=\inf_{\mathscr L(X,\tilde X)\in\Gamma(\mu,\tilde\mu)}{\prob(X\neq \tilde X)}.
\label{Eq_DefTV}\end{equation}
Moreover, we note (for real-valued random variables) $X\leqL \tilde X$ if $\mu((-\infty,x])\geq\tilde\mu((-\infty,x])$ for all $x\in\R$. By a slight abuse of notation, we may use the previous notations for random variables instead of their distributions. It is known that both convergence in $\Wass_p$ and in total variation distance imply convergence in distribution. Observe that any arbitrary coupling provides an upper bound for the left-hand side terms in \eqref{Eq_DefWass} and \eqref{Eq_DefTV}. The classical egality below is easy to show, and will be used later to provide a useful coupling; assuming that $\mu$ and $\tilde\mu$ admit $f$ and $\tilde f$ for respective densities, there exists a coupling $\mathscr L(X,\tilde X)\in\Gamma(\mu,\tilde\mu)$ such that 
\begin{equation}
\prob(X=\tilde X)= \int_\R{f(x)\wedge\tilde f(x)dx}.
\label{Eq_ProbDensite}\end{equation}
Thus,
\begin{equation}
\|\mu-\tilde\mu\|_{TV}= 1-\int_\R{f(x)\wedge\tilde f(x)dx}=\frac12\int_\R{|f(x)-\tilde f(x)|dx}.
\label{Eq_TVDensite}\end{equation}

\subsection{Heuristics}
\label{SSec_Heuristics}
If, given a coupling $(Y,\tilde Y)=\big((X,\Theta,A),(\tilde X,\tilde \Theta,\tilde A)\big)$, we can explicitly control the distance of their distributions at time $t$ regarding their distance at time 0, and if $\mathscr L(\tilde Y_0)$ is the invariant probability measure, then we control the distance between $\mathscr L(Y_t)$ and this distribution. Formally, let $Y=(X,\Theta,A)$ and $\tilde Y=(\tilde X,\tilde \Theta,\tilde A)$ be two PDMPs generated by $\eqref{Eq_GenL}$ such as $Y_0\eqL\mu_0$ and $\tilde Y_0\eqL\tilde\mu_0$. Denote by $\mu$ (resp. $\tilde\mu$) the law of $Y$ (resp. $\tilde Y$). We call coalescing time of $Y$ and $\tilde Y$ the random variable
\[\tau=\inf\{t\geq0:\forall s\geq 0,Y_{t+s}=\tilde Y_{t+s}\}.\]
Note that $\tau$ is not, a priori, a stopping time (w.r.t. the natural filtration of $Y$ and $\tilde Y$), but if $(Y,\tilde Y)$ is a relevant coupling of $\mu$ and $\tilde\mu$, we may be able to make it so. Moreover, it is easy to check from \eqref{Eq_DefTV} that, for $t>0$,
\begin{equation}
\|\mu_0 P_t-\tilde\mu_0 P_t\|_{TV}\leq\prob(Y_t\neq\tilde Y_t)\leq\prob(\tau> t).
\label{Eq_InegCoupl}\end{equation}
As a consequence, the main idea is to fix $t>0$ and to exhibit a coupling $(Y,\tilde Y)$ such that $\prob(\tau\geq t)$ is exponentially decreasing. Let us now present the coupling we shall use to that purpose. The justifications will be given in Sections~\ref{SSec_AgeCoal}, \ref{SSec_WassCoupl} and \ref{SSec_TVCoupl}.
\begin{itemize}
	\item[$\bullet$] \underline{Phase 1: Ages coalescence} (from 0 to $t_1$)\\
		If $X$ and $\tilde X$ jump separately, it is difficult to control their distance, because we can not control the height of their jumps (if $F$ is not trivial). The aim of the first phase is to force the two processes to jump at the same time once; then, it is possible to choose a coupling with exactly the same jump mechanisms, which makes that the first jump is the coalescing time for $A$ and $\hat A$. Moreover, the randomness of $U$ does not affect the strategy anymore afterwards, since it can be the same for both processes. Similarly, the randomness of $\Theta$ does not matter anymore. Finally, note that, if $\zeta$ is constant, it is always possible to make the processes jump at the same time, and the length of this phase exactly follows an exponential law of parameter $\zeta(0)$.
	\item[$\bullet$] \underline{Phase 2: Wasserstein coupling} (from $t_1$ to $t_2$)\\
		Once there is coalescence of the ages, it is time to make $X$ and $\tilde X$ close to each other. Since we can give the same metabolic parameter and the same jumps at the same time for each process, knowing the distance and the metabolic parameter after the jump, the distance is deterministic until the next jump. Consequently, the distance between $X$ and $\tilde X$ at time $s\in[t_1,t_2]$ is
	\[|X_s-\tilde X_s|=|X_{t_1}-\tilde X_{t_1}|\exp\left(-\int_{t_1}^s{\Theta_rdr}\right).\]
	\item[$\bullet$] \underline{Phase 3: Total variation coupling} (from $t_2$ to $t$)\\
		If $X$ and $\tilde X$ are close enough at time $t_2$, which is the purpose of phase 2, we have to make them jump simultaneously - again - but now at the same point. This can be done since $F$ has a density. In this case, we have $\tau\leq t$; if this is suitably done, then $\prob(\tau\leq t)$ is close to 1 and the result is given by \eqref{Eq_InegCoupl}.
\end{itemize}

\begin{figure}[!ht]
\centering
\begin{tikzpicture}[scale=1.5]
\draw[->] (-0.2,0) -- (8.7,0);
\draw[->] (0,-0.2) -- (0,4.2);
\draw[color=blue] (0,2) node[left]{$X_0$};
\draw[color=red] (0,4) node[left]{$\tilde X_0$};
\draw[-|] (0,0)node[anchor=north east]{0} -- (3.8,0)node[anchor=north]{$t_1$};
\draw [decorate,decoration={brace,amplitude=10pt,mirror},yshift=-0.5cm] (0,0) -- (3.8,0) node[midway,yshift=-0.5cm]{\footnotesize{Phase 1}};
\draw[-|] (3.8,0) -- (6.5,0)node[anchor=north]{$t_2$};
\draw [decorate,decoration={brace,amplitude=10pt,mirror},yshift=-0.5cm] (3.8,0) -- (6.5,0) node[midway,yshift=-0.5cm]{\footnotesize{Phase 2}};
\draw[-|] (6.5,0) -- (8.5,0)node[anchor=north]{$t$};
\draw [decorate,decoration={brace,amplitude=10pt,mirror},yshift=-0.5cm] (6.5,0) -- (8.5,0) node[midway,yshift=-0.5cm]{\footnotesize{Phase 3}};
\draw [domain=0:1,smooth,color=blue] plot (\x,{2*exp(2*(-\x))});
\draw [color=blue,dashed] (1,0.271) -- (1,1.5);
\draw [domain=1:3.5,smooth,color=blue] plot (\x,{1.5*exp(1*(1-\x))});
\draw [color=blue,dashed] (3.5,0.123) -- (3.5,1.5);
\draw [domain=0:1.5,smooth,color=red] plot (\x,{4*exp(3*(-\x))});
\draw [color=red,dashed] (1.5,0.044) -- (1.5,2);
\draw [domain=1.5:2.5,smooth,color=red] plot (\x,{2*exp(1.5*(1.5-\x))});
\draw [color=red,dashed] (2.5,0.446) -- (2.5,3);
\draw [domain=2.5:3.5,smooth,color=red] plot (\x,{3*exp(1*(2.5-\x))});
\draw [color=red,dashed] (3.5,1.104) -- (3.5,4) node[right,black]{\footnotesize{First simultaneous jump}};
\draw [domain=3.5:5.5,smooth,color=blue] plot (\x,{1.5*exp(1.5*(3.5-\x))});
\draw [domain=3.5:5.5,smooth,color=red] plot (\x,{4*exp(1.5*(3.5-\x))});
\draw [color=blue,dashed] (5.5,0.075) -- (5.5,1.575);
\draw [color=red,dashed] (5.5,0.199) -- (5.5,1.699);
\draw [domain=5.5:7,smooth,color=blue] plot (\x,{1.575*exp(0.7*(5.5-\x))});
\draw [domain=5.5:7,smooth,color=red] plot (\x,{1.699*exp(0.7*(5.5-\x))});
\draw [color=blue,dashed] (7,0.551) -- (7,2);
\draw [color=red,dashed] (7,0.596) -- (7,2) node[right,black]{\footnotesize{Coalescence}};
\draw [domain=7:8.7,smooth,color=blue] plot (\x,{2*exp(2*(7-\x))});
\draw [domain=7:8.7,smooth,color=red] plot (\x,{2*exp(2*(7-\x))});
\end{tikzpicture}
\caption{Expected behaviour of the coupling.}
\end{figure}
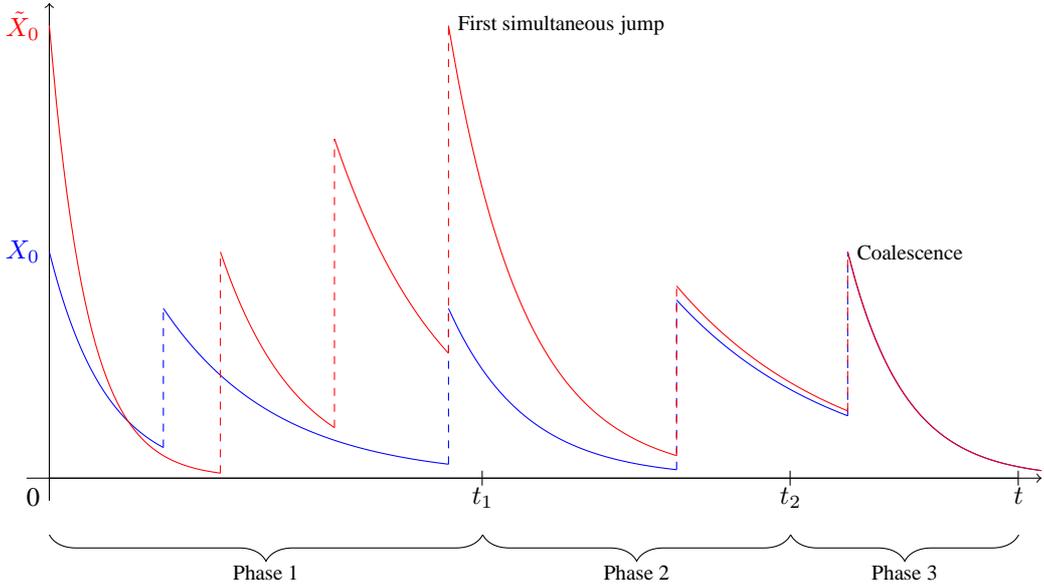

This coupling gives us a good control of the total variation distance of $Y$ and $\tilde Y$, and it can also provide an exponential convergence speed in Wasserstein distance if we set $t_2=t$; this control is expressed with explicit rates of convergence in Theorem~\ref{mainThm}.

\subsection{Ages coalescence}
\label{SSec_AgeCoal}
As explained in Section~\ref{SSec_Heuristics}, we try to bring the ages $A$ and $\tilde A$ to coalescence. Observe that knowing the dynamics of $Y=(X,\Theta,A)$, $A$ is a PDMP with infinitesimal generator
\begin{equation}
\mathcal A\varphi(a)=\partial_a\varphi(a)+\zeta(a)[\varphi(0)-\varphi(a)],
\label{Eq_GenA}\end{equation}
so, for now, we will focus only on the age processes $A$ and $\tilde A$, which is a classical renewal process. The reader may refer to  \cite{Fel71} or \cite{Asm03} for deeper insights about renewal theory. Since $\Delta T_1$ does not follow a priori the distribution $G$, $A$ is a delayed renewal process; anyway this does not affect the sequel, since our method requires to wait for the first jump to occur.

Let $\mu_0,\tilde\mu_0\in\mathscr M(\R_+)$. Denote by $(A,\tilde A)$ the Markov process generated by the following infinitesimal generator:
\begin{equation}
\mathcal A_2\varphi(a,\tilde a)=\partial_a\varphi(a,\tilde a)+\partial_{\tilde a}\varphi(a,\tilde a)+[\zeta(a)-\zeta(\tilde a)][\varphi(0,\tilde a)-\varphi(a,\tilde a)] +\zeta(\tilde a)[\varphi(0,0)-\varphi(a,\tilde a)]
\label{Eq_GenA2}\end{equation}
if $\zeta(a)\geq\zeta(\tilde a)$, and with a symmetric expression if $\zeta(a)<\zeta(\tilde a)$, and such as $A_0\eqL\mu_0$ and $\tilde A_0\eqL\tilde\mu_0$. If $\varphi(a,\tilde a)$ does not depend on $a$ or on $\tilde a$, one can easily check that \eqref{Eq_GenA2} reduces to \eqref{Eq_GenA}, which means that $(A,\tilde A)$ is a coupling of $\mu$ and $\tilde\mu$. Moreover, it is easy to see that, if a common jump occurs for $A$ and $\tilde A$, every following jump will be simultaneous (since the term $\zeta(a)-\zeta(\tilde a)$ will stay equal to 0 in $\mathcal A_2$). Note that, if $\zeta$ is a constant function, then this term is still equal to 0 and the first jump is common. Last but not least, since $\zeta$ is non-decreasing, only two phenomenons can occur: the older process jumps, or both jump together (in particular, if the younger process jumps, the other one jumps as well).

Our goal in this section is to study the time of the first simultaneous jump which will be, as previously mentionned, the coalescing time of $A$ and $\tilde A$; by definition, here, it is a stopping time. Let
\[\tau_A=\inf{\{t\geq0:A_t=\tilde A_t\}}=\inf{\{t\geq0:\forall s\geq0,A_{t+s}=\tilde A_{t+s}\}}.\]
Let
\[\left\{\begin{array}{ll}
a&=\inf\left\{t\geq0:\zeta(t)>0\right\}\in[0,+\infty),\\
d&=\sup\left\{t\geq0:\zeta(t)<+\infty\right\}\in(0,+\infty].
\end{array}\right.\]

\begin{remark}
Note that assumption~\eqref{Eq_H3} guarantees that $\inf\zeta=\zeta(a)$ and $\sup\zeta=\zeta(d^-)$. Moreover, if $d<+\infty$, then $\zeta(d^-)=+\infty$ since $G$ admits a density. Indeed, the following relation is a classical result:
\[\int_0^{\Delta T}{\zeta(s)ds}\eqL \mathscr E(1),\]
which is impossible if $d<+\infty$ and $\zeta(d^-)<+\infty$. A slight generalisation of our model would be to use truncated random variables of the form $\Delta T\wedge C$ for a deterministic constant $C$; then, their common distribution would not admit a density anymore, but the mechanisms of the process would be similar. In that case, it is possible that $d<+\infty$ and $\zeta(d^-)<+\infty$, but the rest of the method remains unchanged.
\end{remark}

First, let us give a good and simple stochastic bound for $\tau_A$ in a particular case.

\begin{Prop}
If $\zeta(0)>0$ then the following stochastic inequality holds:
\[\tau_A\leqL\mathscr E(\zeta(0)).\] 
\end{Prop}

\begin{proof}
It is possible to rewrite $\eqref{Eq_GenA2}$ as follows:
\begin{align*}
\mathcal A_2\varphi(a,\tilde a)=\partial_a\varphi(a,\tilde a)+\partial_{\tilde a}\varphi(a,\tilde a)&+[\zeta(a)-\zeta(\tilde a)][\varphi(0,\tilde a)-\varphi(a,\tilde a)]\\
	&+[\zeta(\tilde a)-\zeta(0)][\varphi(0,0)-\varphi(a,\tilde a)]\\
	&+\zeta(0)[\varphi(0,0)-\varphi(a,\tilde a)],
\end{align*}
for $\zeta(a)\geq\zeta(\tilde a)$. This decomposition of \eqref{Eq_GenA2} indicates that three independent phenomenons can occur for $A$ and $\tilde A$ with respective hazard rates $\zeta(a)-\zeta(\tilde a),\zeta(\tilde a)-\zeta(0)$ and $\zeta(0)$. We have a common jump in the last two cases and, in particular, the inter-arrival times of the latter follow a distribution $\mathscr E(\zeta(0))$ since the rate is constant. Thus, we have $\tau_A\leqL \mathscr E(\zeta(0))$.
\end{proof}

To rephrase this result, the age coalescence occurs stochastically faster than an exponential law. This relies only on the fact that the jump rate is bounded from below, and it is trickier to control the speed of coalescence if $\zeta$ is allowed to be arbitrarily close to 0. This is the purpose of the following theorem.

\begin{Thm}\label{vitAge}
Assume that $\inf\zeta=0$. Let $\varepsilon>\frac a2$. Let $b,c\in(a,d)$ such that $\zeta(b)>0$ and $c>b+\varepsilon$.
\begin{enumerate}[(i)]
	\item If $\frac{3a}{2}<d<+\infty$, then
	\[\tau_A\leqL c+(2H-1)\varepsilon+\sum_{i=1}^H{(d-\varepsilon)G^{(i)}},\]
where $H,G^{(i)}$ are independent random variables of geometric law and $G^{(i)}$ are i.i.d.
	\item If $d=+\infty$ and $\zeta(d^-)<+\infty$, then
	\[\tau_A\leqL \sum_{i=1}^H{\sum_{j=1}^{G^{(i)}}{\left(b+E^{(i,j)}\right)}},\]
where $H,G^{(i)},E^{(i,j)}$ are independent random variables, $G^{(i)}$ are i.i.d. with geometric law, $E^{(i,j)}$ are i.i.d. with exponential law and $\mathscr L(H)$ is geometric.
	\item If $d=+\infty$ and $\zeta(d^-)=+\infty$, then
	\[\tau_A\leqL c-\varepsilon+\sum_{i=1}^H{\left(2\varepsilon+\sum_{j=1}^{G^{(i)}}{\left(c-\varepsilon+E^{(i,j)}\right)}\right)},\]
where $H,G^{(i)},E^{(i,j)}$ are independent random variables, $G^{(i)}$ are i.i.d. with geometric law, $E^{(i,j)}$ are i.i.d. with exponential law and $\mathscr L(H)$ is geometric.
\end{enumerate}
Furthermore, the parameters of the geometric and exponential laws are explicit in terms of the parameters $\varepsilon, a, b, c$ and $d$ (see the proof for details).
\end{Thm}

\begin{remark}\label{Rk_vitAge}
Such results may look technical, but above all they allow us to know that the distribution tail of $\tau_A$ is exponentially decreasing (just like the geometric or exponential laws). If $G$ is known (or equivalently, $\zeta$), Theorem~\ref{vitAge} provides a quantitative exponential bound for the tail. For instance, in case (i), $\tau_A$ admits exponential moments strictly less than $-\frac12\min\left(\frac{\log(1-p_2)}{2\varepsilon},\frac{\log(1-p_1p_2)}{d-\varepsilon}\right)$, since $H$ and $\sum_{i=1}^H{G^{(i)}}$ are (non-independent) random variables with respective exponential moments $-\log(1-p_2)^-$ and $-\log(1-p_1p_2)^-$.
\end{remark}

\begin{remark}
In the case (i), we make the technical assumption that $d\geq\frac{3a}{2}$; this is not compulsory and the results are basically the same, but we cannot use our technique. It comes from the fact that it is really difficult to make the two processes jump together if $d-a$ is small. Without such an assumption, one may use the same arguments with a greater number of jumps, in order to gain room for the jump time of the older process. Provided that the distribution $G$ is spread-out, it is possible to bring the coupling to coalescence (see Theorem~VII.2.7 in \cite{Asm03}) but it is more difficult to obtain quantitative bounds.
\end{remark}

\begin{remark}
Even if Theorem~\ref{vitAge} holds for any set of parameters (recall that $a$ and $d$ are fixed), it can be optimized by varying $\varepsilon,b$ and $c$, depending on $\zeta$. One should choose $\varepsilon$ to be small regarding the length of the jump domain $[b,c]$ (which should be large, but with a small variation of $\zeta$ to maximize the common jump rate).
\end{remark}

\begin{proof}[Proof of Theorem~\ref{vitAge}]
First and foremost, let us prove (i). We recall that the processes $A$ and $\tilde A$ jump necessarily to 0. The method we are going to use here will be applied to the other cases with a few differences. The idea is the following: try to make the distance between $A$ and $\tilde A$ smaller than $\varepsilon$ (which will be called a $\varepsilon$-coalescence), and then make the processes jump together where we can quantify their jump speed (i.e. in a domain where the jump rate is bounded, so that the simultaneous jump is stochastically bounded between two exponential laws). We make the age processes jump together in the domain $[b,c]$, whose length must be greater than $\varepsilon$; since $\varepsilon\geq a/2$ and $[b,c]\subset(a,d)$, this is possible only if $d>\frac{3a}2$. Then, we use the following algorithm:
\begin{itemize}
	\item[$\bullet$] Step 1: Wait for a jump, so that one of the processes (say $\tilde A$) is equal to 0. The length of this step is less than $d<+\infty$ by definition of $d$.
	\item[$\bullet$] Step 2: If there is not yet $\varepsilon$-coalescence (say we are at time $T$), then $A_T>\varepsilon$. We want $A$ to jump before a time $\varepsilon$, so that the next jump implies $\varepsilon$-coalescence. This probability is $1-\exp\left(-\int_0^\varepsilon{\zeta(A_T+s)ds}\right)$, which is greater than the probability $p_1$ that a random variable following an exponential law of parameter $\zeta\left(\varepsilon+\frac a2\right)$ is less than $\varepsilon-\frac a2$. It corresponds to the probability of $A$ jumping between $\frac{a+2\varepsilon}2$ and $2\varepsilon$.
	\item[$\bullet$] Step 3: There is a $\varepsilon$-coalescence. Say $\tilde A=0$ and $A\leq\varepsilon$. Recall that if the younger process jumps, the jump is common. So, if $A$ does not jump before a time $b$, which probability is greater than $\exp\left(-b\zeta(b+\varepsilon)\right)$, and then $\tilde A$ jumps before a time $c-b-\varepsilon$, with a probability greater than $1-\exp\left(-\left(c-b-\varepsilon\right)\zeta(b)\right)$, then coalescence occurs; else go back to Step 2.
\end{itemize}
The previous probabilities can be rephrased with the help of exponential laws:
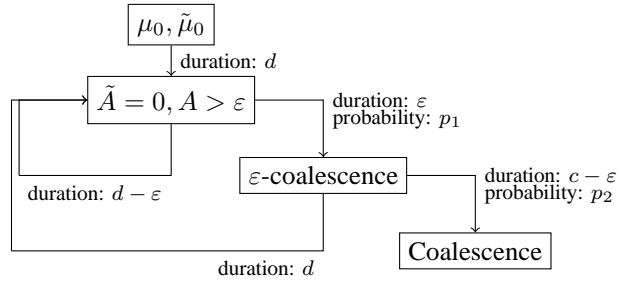
\begin{center}
\begin{tikzpicture}
\node[draw] (A) at (0,0) {$\mu_0,\tilde\mu_0$};
	\node[draw] (B) at (0,-1) {$\tilde A=0,A>\varepsilon$};
	\draw[->] (A) -- node[midway,right]{\footnotesize{duration: $d$}} (B);
		\node[draw] (C) at (2,-2) {$\varepsilon$-coalescence};
		\draw (B) |- node[midway,below left]{\footnotesize{duration: $d-\varepsilon$}} (-2,-2);\draw[->] (-2,-2) |- (B);
		\draw[->] (B) -| node[midway,right]{\footnotesize{duration: $\varepsilon$}} node[midway,below right]{\footnotesize{probability: $p_1$}} (C);
			\node[draw] (D) at (4,-3) {Coalescence};
			\draw (C) |- node[midway,below left]{\footnotesize{duration: $d$}} (-2.1,-3);\draw[->] (-2.1,-3) |- (B);
			\draw[->] (C) -| node[midway,right]{\footnotesize{duration: $c-\varepsilon$}} node[midway,below right]{\footnotesize{probability: $p_2$}} (D);
\end{tikzpicture}
\end{center}
Step 3 leads to coalescence with the help of the arguments mentionned before, using the expression \eqref{Eq_GenA2} of $\mathcal A_2$. Simple computations show that
\[p_1=1-\exp\left(-\left(\varepsilon-\frac a2\right)\zeta\left(\varepsilon+\frac a2\right)\right),\qquad
p_2=\exp\left(-b\zeta(b+\varepsilon)\right)\big(1-\exp\left(-\left(c-b-\varepsilon\right)\zeta(b)\right)\big).\]
Let $G^{(i)}\eqL\mathscr G(p_1)$ be i.i.d. and $H\eqL\mathscr G(p_2)$ Then the following stochastic inequality holds:
\begin{align*}
\tau_A&\leqL d+(d-\varepsilon)(G^{(1)}-1)+\varepsilon+\indic_{\{H\geq2\}}\sum_{i=2}^H{\left(d+(d-\varepsilon)(G^{(i)}-1)+\varepsilon\right)}+(c-\varepsilon)\\
	&\leqL c +(2H-1)\varepsilon+\sum_{i=1}^H{(d-\varepsilon)G^{(i)}}.
\end{align*}
Now, we prove (ii). We make the processes jump simultaneously in the domain $[b,+\infty)$ with the following algorithm:
\begin{itemize}
	\item[$\bullet$] Step 1: Say $A$ is greater than $\tilde A$. We want it to wait for $\tilde A$ to be in domain $[b,+\infty)$. In the worst scenario, it has to wait a time $b$, with a hazard rate less than $\zeta(d^-)<+\infty$. This step lasts less than a geometrical number of times $b$.
	\item[$\bullet$] Step 2: Once the two processes are in the jump domain, two phenomenons can occur: common jump with hazard rate greater than $\zeta(b)$ or jump of the older one with hazard rate less than $\zeta(d^-)$. The first jump occurs with a rate less than $\zeta(d^-)$ and is a simultaneous jump with probability greater than $\frac{\zeta(b)}{\zeta(d^-)}$. If there is no common jump, go back to Step 1.
\end{itemize}
Let
\[p_1=\e^{-b\zeta(d^-)},\qquad p_2=\frac{\zeta(b)}{\zeta(d^-)}.\]
Let $G^{(i)}\eqL\mathscr G(p_1)$ be i.i.d.,$H\eqL\mathscr G(p_2)$ and $E^{(i,j)}\eqL\mathscr E(\zeta(b))$ be i.i.d. Then the following stochastic inequality holds:
\begin{align*}
\tau_A &\leqL \sum_{j=2}^{G^{(1)}}{\left(b+E^{(1,j)}\right)}+b+\indic_{\{H\geq2\}}\sum_{i=2}^H{\left(E^{(i,1)}+\sum_{j=2}^{G^{(i)}}{\left(b+E^{(i,j)}\right)}+b\right)}+E^{(1,1)}\\
	&\leqL \sum_{i=1}^H{\sum_{j=1}^{G^{(i)}}{\left(b+E^{(i,j)}\right)}}.
\end{align*}
Let us now prove (iii). We do not write every detail here, since this case is a combination of the two previous cases (wait for a $\varepsilon$-coalescence, then bring the processes to coalescence using stochastic inequalities involving exponential laws). Let
\[p_1=1-\exp\left(-\left(\varepsilon-\frac a2\right)\zeta\left(\varepsilon+\frac a2\right)\right),\qquad
p_2=\frac{\zeta(b)}{\zeta(c)}\exp\left(-b\zeta(b+\varepsilon)\right)\big(1-\exp\left(-(c-b-\varepsilon)\zeta(b)\right)\big).\]
Let $G^{(i)}\eqL\mathscr G(p_1)$ be i.i.d., $H\eqL\mathscr G(p_2)$ and $E^{(i,j)}\eqL\mathscr E(\zeta(c))$ be i.i.d. Then the following stochastic inequality holds
\begin{align*}
\tau_A&\leqL c+E^{(1,1)}+\varepsilon+\sum_{j=2}^{G^{(1)}}{\left(c-\varepsilon+E^{(1,j)}\right)}+(c-\varepsilon)+\sum_{i=2}^H{\left(c+E^{(i,1)}+\varepsilon+\sum_{j=2}^{G^{(i)}}{\left(c-\varepsilon+E^{(i,j)}\right)}\right)}\\
	&\leqL c-\varepsilon+\sum_{i=1}^H{\left(2\varepsilon+\sum_{j=1}^{G^{(i)}}{\left(c-\varepsilon+E^{(i,j)}\right)}\right)}.
\end{align*}\end{proof}

\subsection{Wasserstein coupling}
\label{SSec_WassCoupl}
Let $\mu_0,\tilde\mu_0\in\mathscr M(\R_+)$. Denote by $(Y,\tilde Y)=(X,\Theta,A,\tilde X,\tilde\Theta,\tilde A)$ the Markov process generated by the following infinitesimal generator:
\begin{align}\nonumber
\mathcal L_2 \varphi(x,\theta,a,\tilde x,\tilde\theta, \tilde a) &=\int_{u=0}^\infty\int_{\theta'=0}^\infty\big([\zeta(a)-\zeta(\tilde a)]\big[\varphi(x+u,\theta',0,\tilde x,\tilde\theta,\tilde a)-\varphi(x,\theta,a,\tilde x,\tilde\theta, \tilde a)\big]\nonumber\\
	&\qquad+\zeta(\tilde a)\big[\varphi(x+u,\theta',0,\tilde x+u,\theta',0)-\varphi(x,\theta,a,\tilde x,\tilde\theta, \tilde a)\big]\big)H(d\theta')F(du)\nonumber\\
	&\qquad-\theta x\partial_x\varphi(x,\theta,a,\tilde x,\tilde\theta, \tilde a)-\tilde\theta \tilde x\partial_x\varphi(x,\theta,a,\tilde x,\tilde\theta, \tilde a)\nonumber\\
	&\qquad+\partial_a\varphi(x,\theta,a,\tilde x,\tilde\theta, \tilde a)+\partial_{\tilde a}\varphi(x,\theta,a,\tilde x,\tilde\theta, \tilde a)
\label{Eq_GenL2}\end{align}
if $\zeta(a)\geq\zeta(\tilde a)$, and with a symmetric expression if $\zeta(a)<\zeta(\tilde a)$, and with $Y_0\eqL\mu_0$ and $\tilde Y_0\eqL\tilde\mu_0$. As in the previous section, one can easily check that $Y$ and $\tilde Y$ are generated by \eqref{Eq_GenL} (so $(Y,\tilde Y)$ is a coupling of $\mu$ and $\tilde\mu$). Moreover, if we choose $\varphi(x,\theta,a,\tilde x,\tilde\theta, \tilde a)=\psi(a,\tilde a)$ then \eqref{Eq_GenL2} reduces to \eqref{Eq_GenA2}, which means that the results of the previous section still hold for the age processes embedded in a coupling generated by \eqref{Eq_GenL2}. As explained in Section~\ref{SSec_AgeCoal}, if $Y$ and $\tilde Y$ jump simultaneously, then they will always jump together afterwards. After the age coalescence, the metabolic parameters and the contaminant quantities are the same for $Y$ and $\tilde Y$. Thus, it is easy to deduce the following lemma, whose proof is straightforward with the previous arguments.

\begin{Lemm}\label{FactExpo}
Let $(Y,\tilde Y)$ be generated by $\mathcal L_2$ in \eqref{Eq_GenL2}. If $A_{t_1}=\tilde A_{t_1}$ and $\Theta_{t_1}=\tilde\Theta_{t_1}$, then, for $t\geq t_1$,
\[A_t=\tilde A_t,\qquad\Theta_t=\tilde \Theta_t.\]
Moreover,
\[|X_t-\tilde X_t|=|X_{t_1}-\tilde X_{t_1}|\exp\left(-\int_{t_1}^t{\Theta_sds}\right).\]
\end{Lemm}

From now on, let $(Y,\tilde Y)$ be generated by $\mathcal L_2$ in \eqref{Eq_GenL2}. We need to control the Wasserstein distance of $X_t$ and $\tilde X_t$; this is done in the following theorem. The reader may refer to \cite{Asm03} for a definition of the direct Riemann-integrability (d.R.i.); one may think at first of "non-negative, integrable and asymptotically decreasing". In the sequel, we denote by $\psi_J$ the Laplace transform of any positive measure $J$: $\psi_J(u)=\int_{\R}{\e^{ux}J(dx)}$.

\begin{Thm}\label{vitWass}
Let $p\geq1$. Assume that $A_0=\tilde A_0$ and $\Theta_0=\tilde\Theta_0$.
\begin{enumerate}[(i)]
	\item If $G=\mathscr E(\lambda)$ (i.e. $\zeta$ is constant, equal to $\lambda$) then,
\begin{equation}
\E\left[\exp\left(-\int_0^t{p\Theta_sds}\right)\right]\leq \exp\left(-\lambda(1-\E\left[\e^{-\Theta_1T_1}\right]) t\right). \label{Eq_vitWassExp}\end{equation}
	\item Let
   \[J(dx)=\E\left[\e^{-p\Theta_1 x}\right]G(dx),\qquad w=\sup\{u\in\R:\psi_J(u)<1\}.\]	
	If $\sup\{u\in\R:\psi_J(u)<1\}=+\infty$, let $w$ be any positive number. Then for all $\varepsilon>0$, there exists $C>0$ such that
	\begin{equation}
	\E\left[\exp\left(-\int_0^t{p\Theta_sds}\right)\right]\leq C\e^{-(w-\varepsilon)t}.
\label{Eq_vitWass}	\end{equation}
	Furthermore, if $\psi_J(w)<1$ and $\psi_G(w)<+\infty$, or if $\psi_J(w)\leq1$ and $t\mapsto\e^{w t}\E\left[\e^{-p\Theta_1t}\right]G((t,+\infty))$ is directly Riemann-integrable, then there exists $C>0$ such that
	\begin{equation}
	\E\left[\exp\left(-\int_0^t{p\Theta_sds}\right)\right]\leq C\e^{-w t}.
\label{Eq_vitWass2}	\end{equation}
\end{enumerate}
\end{Thm}

\begin{remark}\label{Rk_transLaplG}
Note that $w>0$ by \eqref{Eq_H3}, since the probability measure $G$ admits an exponential moment. Indeed, there exist $l,m>0$ such that, for $t\geq l,\zeta(t)\geq m$. Hence $G\leqL l+\mathscr E(m)$, and $\psi_G(u)\leq\e^{ul}+m(m-u)^{-1}<+\infty$ for $u<m$. In particular, if $\sup\zeta=+\infty$, the domain of $\psi_G$ is the whole real line, and \eqref{Eq_vitWass2} holds.
\end{remark}

\begin{remark}
Theorem~\ref{vitWass} provides a speed of convergence to 0 for $\E\left[\exp\left(-\int_0^t{p\Theta_sds}\right)\right]$ when $t\to+\infty$ under various assumptions. To prove it, we turn to the renewal theory (for a good review, see \cite{Asm03}), which has already been widely studied. Here, we link the boundaries we obtained to the parameters of our model.
\end{remark}

\begin{remark}
If $\sup\{u\in\R:\psi_J(u)<1\}=+\infty$, Theorem~\ref{vitWass} asserts that, for any $w>0$, there exists $C>0$ such that $Z\leq C\e^{-w t}$, which means its decay is faster than any exponential rate. Moreover, note that a sufficient condition for $t\mapsto\e^{w t}\E\left[\e^{-p\Theta t}\right]\prob(\Delta T>t)$ to be d.R.i. is that there exists $\varepsilon>0$ such that $\psi_G(w+\varepsilon)<+\infty$. Indeed,
\[\e^{w t}\E[\e^{-p\Theta t}]\prob(\Delta T> t)\leq\e^{w t}\E[\e^{-p\Theta t}]\e^{-(w+\varepsilon) t}\psi_G(w+\varepsilon)\leq\psi_G(w+\varepsilon)\e^{-\varepsilon t},\]
and the right-hand side is d.R.i.
\end{remark}

\begin{proof}[Proof of Theorem~\ref{vitWass}]
In this context, $\mathscr L(\Delta T_1)\leqL G$; it is harmless to assume that $\mathscr L(\Delta T_1)\eqL G$, since this assumptions only slows the convergence down. Then, denote by $\Theta$ and $\Delta T$ two random variables distributed according to $H$ and $G$ respectively. Let us prove (i); in this particular case, since $\zeta$ is constant equal to $\lambda$, $N_t\eqL\mathscr P(\lambda t)$, so

\begin{align*}
\E\left[\exp\left(-\int_0^t{p\Theta_sds}\right)\right]&=\E\left[\exp\left(-\indic_{\{N_t\geq1\}}\sum_{i=1}^{N_t}{p\Theta_i\Delta T_i}-p\Theta_{N_t+1}(t-T_{N_t})\right)\right]\\
	&\leq \E\left[\exp\left(-\indic_{\{N_t\geq1\}}\sum_{i=1}^{N_t}{p\Theta_i\Delta T_i}\right)\right]\\
	&\leq \prob(N_t=0)+\sum_{n=1}^\infty{\E\left[\exp\left(-\sum_{i=1}^{n}{p\Theta_i\Delta T_i}\right)\right]\prob(N_t=n)}\\
	&\leq \e^{-\lambda t}+\sum_{n=1}^\infty{\e^{-\lambda t}\frac{(\lambda t)^n}{n!}\E\left[\e^{-p\Theta\Delta T}\right]^n}\\
	&\leq \exp\left(-\lambda (1-\E[\e^{-p\Theta\Delta T}]) t\right).\\
\end{align*}
Now, let us prove (ii). Let $Z(t)=\E\left[\exp\left(-\int_0^t{p\Theta_sds}\right)\right]$; we have
\begin{align*}
Z(t)&=\E\left[\exp\left(-\int_0^t{p\Theta_sds}\right)\indic_{\{T_1>t\}}\right]+ \E\left[\exp\left(-\int_0^t{p\Theta_sds}\right)\indic_{\{T_1\leq t\}}\right]\\
	&=\E[\e^{-p\Theta t}]\prob(\Delta T>t)+\int_0^t{\E\left[\e^{-p\Theta x}\exp\left(-\int_x^t{p\Theta_sds}\right)\right]G(dx)}\\
	&=\E[\e^{-p\Theta t}]\prob(\Delta T>t)+\int_0^t{\E\left[\e^{-p\Theta x}\right] \E\left[\exp\left(-\int_0^{t-x}{p\Theta_sds}\right)\right]G(dx)}\\
	&= z(t)+J\ast Z(t),
\end{align*}
where $z(t)=\E[\e^{-p\Theta t}]\prob(\Delta T>t)$ and $J(dt)=\E[\e^{-p\Theta t}]G(dt)$. Since $J(\R)<1$, the function $Z$ satisfies the defective renewal equation
\[Z=z+J\ast Z.\]
Let $\varepsilon>0$ ; the function $\psi_J$ is well defined, continuous, non-decreasing on $(-\infty,w)$, and $\psi_J(w-\varepsilon)<1$. Let
\[Z'(t)=\e^{(w-\varepsilon)t}Z(t),\qquad z'(t)=\e^{(w-\varepsilon)t}z(t),\qquad J'(dt)=\e^{(w-\varepsilon)t}J(dt).\]
It is easy to check that $J'\ast Z'(t)=\e^{(w-\varepsilon)t}J\ast Z(t)$, thus $Z'$ satisfies the renewal equation
\begin{equation}
Z'=z'+J'\ast Z',
\label{Eq_RenewalPrime}\end{equation}
which is defective since $J'(\R)=\psi_{J'}(0)=\psi_J(w-\varepsilon)<1$. Let us prove that $\lim_{t\to+\infty}z'(t)=0$. Let
\[v=\sup\{u>0:\psi_G(u)<+\infty\}.\]
Since $G$ admits exponential moments, $v\in(0,+\infty]$. If $w<v$,
\begin{align}
z'(t)&=\e^{(w-\varepsilon)t}\E\left[\e^{-p\Theta t}\right]\prob\left(\e^{w\Delta T}>\e^{w t}\right)\leq \e^{(w-\varepsilon)t}\E\left[\e^{-p\Theta t}\right]\psi_G(w)\e^{-w t}\notag\\
	&\leq \psi_G(w)\e^{-\varepsilon t}\E\left[\e^{-p\Theta t}\right],\label{Eq_zprimeBounded}
\end{align}
then $\lim_{t\to+\infty}z'(t)=0$. If $v\leq w$, temporarily set $\varphi(t)=\E\left[\exp\left((w-2\varepsilon/3-p\Theta-v)t\right)\right]$.
Assume that $\prob(w-2\varepsilon/3-p\Theta-v\geq0)\neq0$. Thus, if $\prob(w-2\varepsilon/3-p\Theta-v>0)>0$, then $\lim_{t\to+\infty}\varphi(t)=+\infty$; else, $\lim_{t\to+\infty}\varphi(t)=\prob(w-2\varepsilon/3-p\Theta-v=0)>0$. Anyway, there exist $t_0,M>0$ such that for all $t\geq t_0,\varphi(t)\geq M$. It implies
\[\int_0^\infty{\varphi(t)\e^{(v+\varepsilon/3)t}g(t)dt}\geq M\int_{t_0}^\infty{\e^{(v+\varepsilon/3)t}g(t)dt}=+\infty,\]
since $\psi_G(v+\varepsilon/3)=+\infty$, which contradicts the fact that
\[\psi_J(w-\varepsilon/3)=\int_0^\infty{\E\left[\exp\left((w-2\varepsilon/3-p\Theta-v)t\right)\right]\e^{(v+\varepsilon/3)t}g(t)dt}<+\infty.\]
Thus, $\prob(w-2\varepsilon/3-p\Theta-v<0)=1$ and $\lim_{t\to+\infty}\varphi(t)=0$. Using the Markov inequality like for \eqref{Eq_zprimeBounded}, we have
\[z'(t)\leq \psi_G(v-\varepsilon/3)\E\left[\exp\left((w-2\varepsilon/3-p\Theta-v)t\right)\right] =\psi_G(v-\varepsilon/3)\varphi(t),\]
from which we deduce $\lim_{t\to+\infty}z'(t)=0$. Using Proposition V.7.4 in \cite{Asm03}, $Z'$ is bounded, so there exists $C>0$ such that \eqref{Eq_vitWass} holds. From \cite{Asm03}, note that the function $Z'$ can be explicitly written as $Z'=\left(\sum_{n=0}^\infty{(J')^{\ast n}}\right)\ast z'$. Using this expression, it is possible to make $C$ explicit, or at least to approximate it with numerical methods.

Eventually, we look at \eqref{Eq_RenewalPrime} in the case $\varepsilon=0$. First, if $\psi_J(w)<1$ and $\psi_G(w)<+\infty$, it is straightforward to apply the previous argument (since \eqref{Eq_RenewalPrime} remains defective and \eqref{Eq_zprimeBounded} still holds). Next, if $\psi_J(w)\leq1$ and $z':t\mapsto\e^{w t}z(t)$ is d.R.i., we can apply Theorem V.4.7 - the Key Renewal Theorem - or Proposition V.7.4 in \cite{Asm03}, whether $\psi_J(w)=1$ or $\psi_J(w)<1$. As a consequence, $Z'$ is still bounded, and there still exists $C>0$ such that \eqref{Eq_vitWass2} holds.
\end{proof}

The following corollary is of particular importance because it allows us to control the Wasserstein distance of the processes $X$ and $\tilde X$ defined in \eqref{Eq_DefWass}.

\begin{Coro}\label{vitWass2}
Let $p\geq1$. Assume that $A_{t_1}=\tilde A_{t_1},\Theta_{t_1}=\tilde\Theta_{t_1}$.
\begin{enumerate}[(i)]
	\item There exist $v>0,C>0$ such that, for $t\geq t_1$,
\[\Wass_p(X_t,\tilde X_t)\leq C\exp\left(-v(t-t_1)\right)\Wass_p(X_{t_1},\tilde X_{t_1}).\]
	\item Furthermore, if $\zeta$ is a constant equal to $\lambda$ then, for $t\geq t_1$,
\[\Wass_p(X_t,\tilde X_t)\leq \exp\left(-\frac\lambda p (1-\E[\e^{-p\Theta_1T_1}]) (t-t_1)\right)\Wass_p(X_{t_1},\tilde X_{t_1}).\]
\end{enumerate}
\end{Coro}

\begin{proof}
By Markov property, assume w.l.o.g. that $t_1=0$. Under the notations of Theorem~\ref{vitWass}, note $v=p^{-1}(w-\varepsilon)$ for $\varepsilon>0$, or even $v=p^{-1}w$ if $\psi_J(w)<1$ and $\psi_G(w)<+\infty$, or $t\mapsto\e^{w t}\E\left[\e^{-p\Theta t}\right]\prob(\Delta T>t)$ is directly Riemann-integrable. Thus, (i) follows straightforwardly from \eqref{Eq_vitWass} or \eqref{Eq_vitWass2} using Lemma~\ref{FactExpo}. Relation (ii) is obtained similarly from \eqref{Eq_vitWassExp}.
\end{proof}

\subsection{Total variation coupling}
\label{SSec_TVCoupl}
Quantitative bounds for the coalescence of $X$ and $\tilde X$, when $A$ and $\tilde A$ are equal and $X$ and $\tilde X$ are close, are provided in this section. We are going to use assumption~\eqref{Eq_H1}, which is crucial for our coupling method. Recall that we denote by $f$ the density of $F$, which is the distribution of the jumps $U_n=X_{T_n}-X_{T_n^-}$. From \eqref{Eq_TVDensite}, it is useful to set, for small $\varepsilon$,
\begin{equation}
\eta(\varepsilon)=1-\int_\R{f(x)\wedge f(x-\varepsilon)dx}=\frac12\int_\R{\left|f(x)-f(x-\varepsilon)\right|dx}.
\label{Eq_DefEta}\end{equation}

\begin{figure}[!ht]
\centering
\begin{tikzpicture}[xscale=2.5,yscale=6]
\fill[domain=0.001:1.999,smooth,color=gray!40] plot (\x,{exp(-1/(1-(\x-1)*(\x-1)))});
\fill[domain=0.701:2.699,smooth,color=white] plot (\x,{exp(-1/(1-(\x-1.7)*(\x-1.7)))}) node[above right]{};
\draw[->] (-0.2,0) -- (2.9,0) node[right]{$x$};
\draw[->] (0,-0.2) -- (0,0.5);
\draw[domain=0.001:1.999,smooth,color=blue] plot (\x,{exp(-1/(1-(\x-1)*(\x-1)))}) node[above right]{$f(x)$};
\draw[domain=0.701:2.699,smooth,color=red] plot (\x,{exp(-1/(1-(\x-1.7)*(\x-1.7)))}) node[above right]{$f(x-\varepsilon)$};
\draw[circle,gray] (0.7,0.3)node[anchor=south]{$\eta(\varepsilon)$};
\draw[-|] (0,0)node[anchor=north east]{0} -- (0.7,0)node[anchor=north]{$\varepsilon$};
\end{tikzpicture}
\caption{Typical graph of $\eta$.}
\end{figure}
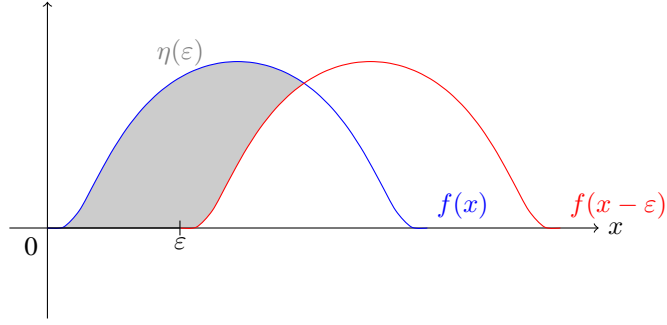

\begin{Def}
Assume that $A_t=\tilde A_t$. We call "TV coupling" the following coupling:
\begin{itemize}
	\item[$\bullet$] From $t$, let $(Y,\tilde Y)$ be generated by $\mathcal L_2$ in \eqref{Eq_GenL2} and make $Y$ and $\tilde Y$ jump at the same time (say $T$).
	\item[$\bullet$] Then, knowing $(Y_{T^-},\tilde Y_{T^-})$, use the coupling provided by \eqref{Eq_ProbDensite} for $X_{T-}+U$ and $\tilde X_{T^-}+\tilde U$.
\end{itemize}
\end{Def}

With the previous notations, conditioning on $\{X_{T^-},\tilde X_{T^-}\}$, it is straightforward that $\prob(X_T=\tilde X_T)\geq 1-\eta\left(\left|X_{T^-}-\tilde X_{T^-}\right|\right)$. Let
\[\tau=\inf\{u\geq0:\forall s\geq u,Y_s=\tilde Y_s\}\]
be the coalescing time of $Y$ and $\tilde Y$; from \eqref{Eq_TVDensite} and \eqref{Eq_DefEta}, one can easily check the following proposition.

\begin{Prop}
\label{vitTV}
Let $\varepsilon>0$. Assume that $A_{t_2}=\tilde A_{t_2}$, $\Theta_{t_2}=\tilde \Theta_{t_2}$ and $|X_{t_2}-\tilde X_{t_2}|\leq\varepsilon$. If $(Y,\tilde Y)$ follows the TV coupling, then
\[\prob\left(X_{T_{N_{t_2}+1}}\neq\tilde X_{T_{N_{t_2}+1}}\right)\leq \sup_{x\in[0,\varepsilon]}\eta(x).\]
\end{Prop}

This proposition is very important, since it enables us to quantify the probability to bring $X$ and $\tilde X$ to coalescence (for small $\varepsilon$), and then $(X,\Theta,A)$ and $(\tilde X,\tilde\Theta,\tilde A)$. With good assumptions on the density $f$ (typically \eqref{Eq_H4a} or \eqref{Eq_H4b}), one can also easily control the term $\sup_{x\in[0,\varepsilon]}\eta(x)$; this is the point of the lemma below.

\begin{Lemm}
\label{majSupEta}
Let $0<\varepsilon<1$. There exist $C,v>0$ such that
\begin{equation}
\sup_{x\in[0,\varepsilon]}\eta(x)\leq C\varepsilon^{v}.
\label{Eq_majSupEta}\end{equation}
\end{Lemm}

\begin{proof}
Assumptions~\eqref{Eq_H4a} and \eqref{Eq_H4b} are crucial here. If \eqref{Eq_H4a} is fullfiled, which means $\eta$ is Hölder, \eqref{Eq_majSupEta} is straightforward (and $v$ is its Hölder exponent, since $\eta(0)=0$). Otherwise, assume that \eqref{Eq_H4b} is true: $f$ is $h$-Hölder, that is to say there exist $K,h>0$ such that $|f(x)-f(y)|<K|x-y|^h$, and $\lim_{x\to+\infty}x^pf(x)=0$ for some $p>2$. Then, denote by $D_\varepsilon$ the $(1-\varepsilon^h)$-quantile of $F$, so that
\[\int_{D_\varepsilon}^\infty{f(u)du}=\varepsilon^h.\]
Then, we have, for all $x\leq\varepsilon$,
\begin{align}
\eta(x)&=\frac12\left(\int_0^{D_\varepsilon+1}{|f(u)-f(u-x)|du}+\int_{D_\varepsilon+1}^\infty{|f(u)-f(u-x)|du}\right)\notag\\
	&\leq\frac12\int_0^{D_\varepsilon+1}{|f(u)-f(u-x)|du}+\frac12\int_{D_\varepsilon+1}^\infty{(f(u)+f(u-x))du}\notag\\
	&\leq \left(K\frac{D_\varepsilon+1}2+1\right)\varepsilon^h.\label{Eq_majSupEta2}
\end{align}
Now, let us control $D_\varepsilon$; there exists $C'>0$ such that $f(x)\leq C'x^{-p}$. Then,
\[\int_{\left(\frac{C'}{(p-1)\varepsilon^h}\right)^\frac1{p-1}}^\infty{f(x)dx}\leq\int_{\left(\frac{(p-1)\varepsilon^h}{C'}\right)^\frac{-1}{p-1}}^\infty{C'x^{-p}dx}=\varepsilon^h,\]
so
\begin{equation}
D_\varepsilon\leq\left(\frac{C'}{(p-1)\varepsilon^h}\right)^\frac1{p-1}.
\label{Eq_majSupEta3}
\end{equation}
Denoting by
\[C=K\frac{\left(\frac{C'}{p-1}\right)^\frac1{p-1}+1}2+1,\qquad v=h-\frac h{p-1},\]
the parameter $v$ is positive because $p>2$, and \eqref{Eq_majSupEta} follows from \eqref{Eq_majSupEta2} and \eqref{Eq_majSupEta3}.
\end{proof}

\section{Main results}
\label{Sec_Results}
In this section, we use the tools provided in Section~\ref{Sec_Speeds} to bound the coalescence time of the processes and prove the main result of this paper, Theorem~\ref{mainThm}; some better results are also derived in a specific case. Two methods will be presented. The first one is general and may be applied in every case, whereas the second one uses properties of homogeneous Poisson processes, which is relevant only in the particular case where the inter-intake times follow an exponential distribution, and, a priori, cannot be used in other cases. From now on, let $Y$ and $\tilde Y$ be two PDMPs generated by $\mathcal L$ in \eqref{Eq_GenL}, with $\mathscr L(Y_0)=\mu_0$ and $\mathscr L(\tilde Y_0)=\tilde\mu_0$. Let $t$ be a fixed positive real number, and, using \eqref{Eq_InegCoupl}, we aim at bounding $\prob(\tau> t)$ from above ; recall that $\tau_A$ and $\tau$ are the respective coalescing times of the PDMPs $A$ and $\tilde A$, and $Y$ and $\tilde Y$. The heuristic is the following: the interval $[0,t]$ is splitted into three domains, where we apply the three results of Section~\ref{Sec_Speeds}.
\begin{itemize}
	\item[$\bullet$] First domain: apply the strategy of Section~\ref{SSec_AgeCoal} to get age coalescence.
	\item[$\bullet$] Second domain: move $X$ and $\tilde X$ closer with $\mathcal L_2$, as defined in Section~\ref{SSec_WassCoupl}.
	\item[$\bullet$] Third domain: make $X$ and $\tilde X$ jump at the same point, using the density of $F$ and the TV coupling of Section~\ref{SSec_TVCoupl}.
\end{itemize}

\subsection{A deterministic division}
\label{SSec_DetermDivision}
The coupling method we present here bounds from above the total variation distance of the processes. The division of the interval $[0,t]$ will be deterministic, whereas it will be random in Section~\ref{SSec_PartCase}. To this end, let $0<\alpha<\beta<1$. The three domains will be $[0,\alpha t],(\alpha t,\beta t]$ and $(\beta t,t]$. Now, we are able to prove Theorem~\ref{mainThm}. Recall that
\[\tau=\inf\{t\geq0:\forall s\geq 0,Y_{t+s}=\tilde Y_{t+s}\}\]
is the coalescing time of $Y$ and $\tilde Y$, and $\tau_A$ is the coalescing time of $A$ and $\tilde A$.

\begin{proof}[Proof of Theorem~\ref{mainThm}, (i):]
Let $\varepsilon>0$. Let $(Y,\tilde Y)$ be the coupling generated by $\mathcal L_2$ in \eqref{Eq_GenL2} on $[0,\beta t]$ and the TV coupling on $(\beta t,t]$. Let us compute the probabilities of the following tree:

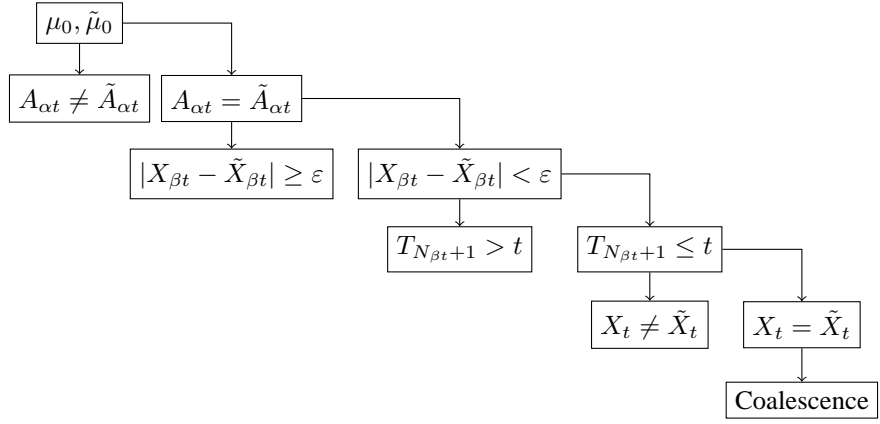
\begin{center}\begin{tikzpicture}
\node[draw] (A) at (0,0) {$\mu_0,\tilde\mu_0$};
	\node[draw] (B) at (0,-1) {$A_{\alpha t}\neq\tilde A_{\alpha t}$};
	\draw[->] (A) -- (B);
	\node[draw] (C) at (2,-1) {$A_{\alpha t}=\tilde A_{\alpha t}$};
	\draw[->] (A) -| (C);
		\node[draw] (D) at (2,-2) {$|X_{\beta t}-\tilde X_{\beta t}|\geq\varepsilon$};
		\draw[->] (C) -- (D);
		\node[draw] (E) at (5,-2) {$|X_{\beta t}-\tilde X_{\beta t}|<\varepsilon$};
		\draw[->] (C) -| (E);
			\node[draw] (F) at (5,-3) {$T_{N_{\beta t}+1}> t$};
			\draw[->] (E) -- (F);
			\node[draw] (G) at (7.5,-3) {$T_{N_{\beta t}+1}\leq t$};
			\draw[->] (E) -| (G);
				\node[draw] (H) at (7.5,-4) {$X_t\neq\tilde X_t$};
				\draw[->] (G) -- (H);
				\node[draw] (I) at (9.5,-4) {$X_t=\tilde X_t$};
				\draw[->] (G) -| (I);
					\node[draw] (J) at (9.5,-5) {Coalescence};
					\draw[->] (I) -- (J);
\end{tikzpicture}\end{center}

Recall from \eqref{Eq_InegCoupl} that $\|\mu_0P_t-\mu_0P_t\|_{TV}\leq\prob(\tau>t)$. Thus,
\begin{multline}
\prob(\tau\leq t)\geq\prob\left(\tau_A\leq\alpha t\right)
\prob\left(\left.|X_{\beta t}-\tilde X_{\beta t}|<\varepsilon\right|\tau_A\leq\alpha t\right)
\prob\left(\left.T_{N_{\beta t}+1}\leq t\right|\tau_A\leq\alpha t,|X_{\beta t}-\tilde X_{\beta t}|<\varepsilon\right)\\
\prob\left(\left.\tau\leq t\right|\tau_A\leq\alpha t,|X_{\beta t}-\tilde X_{\beta t}|<\varepsilon,T_{N_{\beta t}+1}\leq t\right).
\label{Eq_Meth1}
\end{multline}
First, by Theorem~\ref{vitAge}, we know that the distribution tail of $\tau_A$ is exponentially decreasing, since $\tau_A$ is a linear combination of (non-independent) exponential and geometric random variables. Therefore,
\[\prob\left(\tau_A>\alpha t\right)\leq C_1\e^{-v_1\alpha t},\]
where the parameters $C_1$ and $v_1$ are directly provided by Theorem~\ref{vitAge} (see Remark~\ref{Rk_vitAge}). Now, conditioning on $\{\tau_A\leq t\}$, using Corollary~\ref{vitWass2}, there exist $C_2',v_2'>0$ such that
\[\prob\left(\left.|X_{\beta t}-\tilde X_{\beta t}|\geq\varepsilon\right|\tau_A\leq\alpha t\right) \leq\frac{\Wass_1(X_{\beta t},\tilde X_{\beta t})}{\varepsilon}\leq \frac{\Wass_1(X_{\alpha t},\tilde X_{\alpha t})}{\varepsilon}C_2'\e^{-v_2'(\beta-\alpha)t}.\]
Let $U,\Delta T,\Theta$ be independent random variables of respective laws $F,G,H$, and say that any sum between $i$ and $j$ is equal to zero if $i>j$. We have
\begin{align*}
\E\left[X_{\alpha t}\right]&\leq\E\left[X_{T_{N_{\alpha t}}}\right]\leq
\E\left[X_0\exp\left(-\sum_{k=2}^{N_{\alpha t}}{\Theta_k\Delta T_k}\right)+\sum_{i=1}^{N_{\alpha t}}{U_i\exp\left(-\sum_{k=i+1}^{N_{\alpha t}}{\Theta_k\Delta T_k}\right)}\right]\\
&\leq \prob(N_{\alpha t}=0)\E[X_0]+\sum_{n=1}^\infty{\prob(N_{\alpha t}=n)\left(\E[X_0]\E\left[\e^{-\Theta\Delta T}\right]^{n-1}+\E[U]\sum_{k=0}^{n-1}{\E\left[\e^{-\Theta\Delta T}\right]^k}\right)}\\
&\leq\E[X_0]+\sum_{n=0}^\infty{\prob(N_{\alpha t}=n)\left(\frac{\E[X_0]\E\left[\e^{-\Theta\Delta T}\right]^n}{\E\left[\e^{-\Theta\Delta T}\right]}+\E[U]\frac{1-\E\left[\e^{-\Theta\Delta T}\right]^n}{1-\E\left[\e^{-\Theta\Delta T}\right]}\right)}\\
&\leq\E[X_0]+\sum_{n=0}^\infty{\prob(N_{\alpha t}=n)\left(\frac{\E[X_0]}{\E\left[\e^{-\Theta\Delta T}\right]}+\frac{\E[U]}{1-\E\left[\e^{-\Theta\Delta T}\right]}\right)}\\
&\leq \E[X_0]\left(1+\frac{1}{\E\left[\e^{-\Theta\Delta T}\right]}\right)+\frac{\E[U]}{1-\E\left[\e^{-\Theta\Delta T}\right]}.
\end{align*}
Hence,
\begin{align*}
\Wass_1(X_{\alpha t},\tilde X_{\alpha t})&\leq \E\left[X_{\alpha t}\vee\tilde X_{\alpha t}\right]\leq \E\left[X_{\alpha t}\right]+\E\left[\tilde X_{\alpha t}\right]\\
&\leq (\E[X_0+\tilde X_0])\left(1+\frac{1}{\E\left[\e^{-\Theta\Delta T}\right]}\right)+\frac{2\E[U]}{1-\E\left[\e^{-\Theta\Delta T}\right]}.
\end{align*}
Note $C_2=\left((\E[X_0+\tilde X_0])\left(1+\frac{1}{\E\left[\e^{-\Theta\Delta T}\right]}\right)+\frac{2\E[U]}{1-\E\left[\e^{-\Theta\Delta T}\right]}\right)C_2'$. Recall that $G$ admits an exponenital moment (see Remark~\ref{Rk_transLaplG}). We have, using the Markov property, for all $v_3$ such that $\psi_G(v_3)<+\infty$:
\[\prob\left(\left.T_{N_{\beta t}+1}> t\right|\tau_A\leq\alpha t,|X_{\beta t}-\tilde X_{\beta t}|<\varepsilon\right)\leq\prob\left(\Delta T>(1-\beta)t\right)\leq\psi_G(v_3)\e^{-v_3(1-\beta)t}.\]
Note $C_3=\psi_G(v_3)$. Using Proposition~\ref{vitTV} and Lemma~\ref{majSupEta}, we have
\[\prob\left(\left.\tau> t\right|\tau_A\leq\alpha t,|X_{\beta t}-\tilde X_{\beta t}|<\varepsilon,T_{N_{\beta t}+1}\leq t\right)\leq\sup_{x\in[0,\varepsilon]}\eta(x)\leq C_4\varepsilon^{v_4'}.\]
The last step is to choose a correct $\varepsilon$ to have exponential convergence for both the terms $\varepsilon^{-1}C_2 \e^{-v_2'(\beta-\alpha)t}$ and $C_4\varepsilon^{v_4'}$. The natural choice is to fix $\varepsilon=\e^{-v'(\beta-\alpha)t}$, for any $v'<v_2'$. Then, denoting by
\[v_2=v_2'-v',\qquad v_4=v_4'v',\]
and using the egalities above, it is straightforward that \eqref{Eq_Meth1} reduces to \eqref{Eq_MainTV}.
\end{proof}

\begin{remark}\label{Rk_OptimEpsilon}
Theorem~\ref{mainThm} is very important and, above all, states that the exponential rate of convergence in total variation of the PDMP is larger than $\min(\alpha v_1,(\beta-\alpha)v_2,(1-\beta)v_3,(\beta-\alpha)v_4)$. If we choose
\[v'=\frac{v_2'}{1+v_4'}\]
in the proof above, the parameters $v_2$ and $v_4$ are equal; then, in order to have the maximal rate of convergence, one has to optimize $\alpha$ and $\beta$ depending on $v_1,v_2,v_3$.
\end{remark}

\begin{proof}[Proof of Theorem~\ref{mainThm}, (ii):]
Let $(Y,\tilde Y)$ be the coupling generated by $\mathcal L_2$ in \eqref{Eq_GenL2}. Note that
\[\Wass_1(Y_t,\tilde Y_t)\leq\E\left[\|(X_t,\Theta_t,A_t)-(\tilde X_t,\tilde\Theta_t,\tilde A_t)\|\right]=\E[|X_t-\tilde X_t|]+\E[|\Theta_t-\tilde \Theta_t|]+\E[|A_t-\tilde A_t|].\]
Recall that $\E\left[X_{\alpha t}\right]\leq\E[X_0]\left(1+\frac1{\E\left[\e^{-\Theta\Delta T}\right]}\right)+\frac{\E[U]}{1-\E\left[\e^{-\Theta\Delta T}\right]}$, and so does $X_t$. The proof of the inequality below follows the guidelines of the proof of (i), using both Remark~\ref{Rk_vitAge} and Corollary~\ref{vitWass2}, which provide respectively the positive constants $C_1',v_1$ and $C_2',v_2$.
\begin{align*}
\Wass_1(X_t,\tilde X_t)&\leq\E\left[|X_t-\tilde X_t|\right]\leq\E\left[\left.|X_t-\tilde X_t|\right|\tau_A> t\right]\prob(\tau_A> t)+\E\left[\left.|X_t-\tilde X_t|\right|\tau_A\leq t\right]\prob(\tau_A\leq t)\\
&\leq\left((\E[X_0+\tilde X_0])\left(1+\frac{1}{\E\left[\e^{-\Theta\Delta T}\right]}\right)+\frac{2\E[U]}{1-\E\left[\e^{-\Theta\Delta T}\right]}\right)\prob(\tau_A> t)+\E\left[\left.|X_t-\tilde X_t|\right|\tau_A\leq t\right]\\
&\leq\left((\E[X_0+\tilde X_0])\left(1+\frac{1}{\E\left[\e^{-\Theta\Delta T}\right]}\right)+\frac{2\E[U]}{1-\E\left[\e^{-\Theta\Delta T}\right]}\right)\left(C_1'\e^{-v_1t}+C_2'\e^{-v_2t}\right).
\end{align*}
It is easy to see that
\[\E\left[\left.|\Theta_t-\tilde \Theta_t|\right|\tau_A> t\right]\leq\E[\Theta_{N_t+1}]+\E[\tilde\Theta_{\tilde N_t+1}]\leq 2\E[\Theta],\]
and that
\[\E\left[\left|A_t-\tilde A_t|\right|\tau_A> t\right]\leq\E[\Delta T_{N_t+1}]+\E[\tilde\Delta \tilde T_{\tilde N_t+1}]\leq 2\E[\Delta T].\]
Finally, we can conclude by writing that
\begin{align*}
\Wass_1(Y_t,\tilde Y_t)&\leq\E\left[\left.|Y_t-\tilde Y_t|\right|\tau_A> t\right]\prob(\tau_A> t)+\E\left[\left.|Y_t-\tilde Y_t|\right|\tau_A\leq t\right]\prob(\tau_A\leq t)\\
&\leq C_1\e^{-v_1t}+C_2\e^{-v_2t},
\end{align*}
denoting by
\[C_1=\left((\E[X_0+\tilde X_0])\left(1+\frac{1}{\E\left[\e^{-\Theta\Delta T}\right]}\right)+\frac{2\E[U]}{1-\E\left[\e^{-\Theta\Delta T}\right]}+2\E[\Theta]+2\E[\Delta T]\right)C_1',\]
and by
\[C_2=\left((\E[X_0+\tilde X_0])\left(1+\frac{1}{\E\left[\e^{-\Theta\Delta T}\right]}\right)+\frac{2\E[U]}{1-\E\left[\e^{-\Theta\Delta T}\right]}\right)C_2'.\]
\end{proof}

\begin{remark}
Proving the convergence in Wasserstein distance in \eqref{Eq_MainWass} is quite easier than the convergence in total variation, and may still be improved by optimizing in $\alpha$. Moreover, it does not require any assumption on $F$ but a finite expectation, thus holds under assumptions~\eqref{Eq_H2} and \eqref{Eq_H3} only.
\end{remark}

Note that we could also use a mixture of the Wasserstein distance for $X$ and $\tilde X$, and the total variation distance for the second and third components, as in \cite{BLBMZ12}; indeed, the processes $\Theta$ and $\tilde\Theta$ on the one hand, and $A$ and $\tilde A$ on the other hand are interesting only when they are equal, i.e. when their distance in total variation is equal to 0.

\subsection{Exponential inter-intake times}
\label{SSec_PartCase}
We turn to the particular case where $G=\mathscr E(\lambda)$ and $f$ is Hölder with compact support, and we present another coupling method with a random division of the interval $[0,t]$. As highlighted above, the assumption on $G$ is not relevant in a dietary context, but offers very simple and explicit rates of convergence. The assumption on $f$ is pretty mild, given that this function represents the intakes of some chemical. It is possible, a priori, to deal easily with classical unbounded distributions the same way (like exponential or $\chi^2$ distributions, provided that $\eta$ is easily computable). We will not treat the convergence in Wasserstein distance (as in Theorem~\ref{mainThm}, (ii)), since the mechanisms are roughly the same.

We provide two methods to bound the rate of convergence of the process in this particular case. On the one hand, the first method is a slight refinement of the speeds we got in Theorem~\ref{mainThm}, since the laws are explicit. On the other hand, we notice that the law of $N_t$ is known and explicit calculations are possible. Thus, we do not split the interval $[0,t]$ into deterministic areas, but into random areas: $[0,T_1],[T_1,T_{N_t}],[T_{N_t},t]$.

Firstly, let
\[\rho=1-\E\left[\e^{-\Theta_1T_1}\right].\]
Using the same arguments as in the proof of Lemma~\ref{majSupEta}, one can easily see that
\begin{equation}
\sup_{x\in[0,\varepsilon]}\eta(x)\leq K\frac{M+1}2\varepsilon^{h},
\label{Eq_majSupEta4}
\end{equation}
if $|f(x)-f(y)|\leq K|x-y|^h$ and $f(x)=0$ for $x>M$.

\begin{Prop}
For $\alpha,\beta\in(0,1),\alpha<\beta$,
\begin{multline}
\|\mu_0 P_t-\tilde\mu_0 P_t\|_{TV}\leq 1-\left(1-\e^{-\lambda\alpha t}\right)\left(1-\e^{-\lambda(1-\beta)t}\right) \left(1-C\e^{-\frac{\lambda\rho h}{1+h}(\beta-\alpha)t}\right)\\
\left(1-K\frac{M+1}2\e^{-\frac{\lambda\rho h}{1+h}(\beta-\alpha)t}\right),
\notag
\end{multline}
where $C=(\E[X_0+\tilde X_0])\left(1+\frac{1}{1-\rho}\right)+\frac{2\E[U]}{\rho}$.
\end{Prop}

We do not give the details of the proof because they are only slight refinements of the bounds in \eqref{Eq_Meth1}, with parameter $\varepsilon=\exp\left(-\frac{\lambda\rho(\beta-\alpha)}{1+h}t\right)$, since the rates of convergence are $v_2'=\lambda\rho$ and $v_4'=h$. This choice optimizes the speed of convergence, as highlighted in Remark~\ref{Rk_OptimEpsilon}. Note that the constant $C$ could be improved since $\psi_{N_{\alpha t}}$ is known, but this is a detail which does not change the rate of convergence. Anyway, we can optimize these bounds by setting $\beta=1-\alpha$ and $\alpha=\frac{\rho h}{1+h+2\rho h}$, so that the following inequality holds:
\begin{multline}
\|\mu_0 P_t-\tilde\mu_0 P_t\|_{TV}\leq 1-\left(1-\exp\left(\frac{-\lambda\rho h}{1+h+2\rho h}t\right)\right)^2 \left(1-C\exp\left(\frac{-\lambda\rho h}{1+h+2\rho h}t\right)\right)\\
\left(1-K\frac{M+1}2\exp\left(\frac{-\lambda\rho h}{1+h+2\rho h}t\right)\right).
\notag
\end{multline}
Then, developping the previous quantity, there exists $C_1>0$ such that
\begin{equation}
\|\mu_0 P_t-\tilde\mu_0 P_t\|_{TV}\leq C_1\exp\left(\frac{-\lambda\rho h}{1+h+2\rho h}t\right).
\label{Eq_VitPartCase1}
\end{equation}

Before exposing the second method, the following lemma is based on standard properties of the homogeneous Poisson processes, that we recall here.

\begin{Lemm}
Let $N$ be a homogeneous Poisson process of intensity $\lambda$.
\begin{enumerate}[(i)]
	\item $N_t\eqL\mathscr P(\lambda t)$.
	\item $\mathscr L(T_1,T_2,\dots,T_n|N_t=n)$ has a density $(t_1,\dots,t_n)\mapsto t^{-n}n!\indic_{\{0\leq t_1\leq t_2\leq\dots\leq t_n\leq t\}}$.
	\item $\mathscr L(T_1,T_n|N_t=n)$ has a density $g_n(u,v)=t^{-n}n(n-1)(v-u)^{n-2}\indic_{\{0\leq u\leq v\leq t\}}$.
\end{enumerate}
\end{Lemm}

Since $\mathscr L(T_1,T_n|N_t=n)$ is known, it is possible to provide explicit and better results in this specific case.

\begin{Prop}
\label{vitExpo}
For all $\varepsilon<1$, the following inequality holds:
\begin{multline}
\|\mu_0 P_t-\tilde\mu_0 P_t\|_{TV}\leq 1-\left(1- \e^{-\lambda t}\left(1+\lambda t+\frac{\E[X_0\vee\tilde X_0]}{\varepsilon(1-\rho)^2}\left(\e^{\lambda(1-\rho)t}-1-\lambda(1-\rho)t\right)\right)\right)\\
\left(1-K\frac{M+1}2\varepsilon^{h}\right).
\notag
\end{multline}
\end{Prop}

\begin{proof}
Let $0<\varepsilon<1$ and $(Y,\tilde Y)$ be the coupling generated by $\mathcal L_2$ in \eqref{Eq_GenL2} between $0$ and $T_{N_t-1}$ and be the TV coupling between $T_{N_t-1}$ and $t$.
If $n\geq2$ then, knowing $\{N_t=n\}$,
\begin{align*}
\prob(|X_{T_n^-}-\tilde X_{T_n^-}|\geq\varepsilon)&\leq\frac{1}{\varepsilon}\E[|X_{T_n^-}-\tilde X_{T_n^-}|] \leq\frac1\varepsilon\iint_{\R^2}{\E\left[\left.|X_{T_n^-}-\tilde X_{T_n^-}|\right|T_1=u,T_n=v\right]g_n(u,v)dudv}\\
	&\leq\frac{n(n-1)\E[X_0\vee\tilde X_0]}{\varepsilon t^n}\iint_{[0,t]^2}{\e^{-\lambda\rho(v-u)}(v-u)^{n-2}\indic_{\{u\leq v\}}dudv}\\
	&\leq\frac{n(n-1)\E[X_0\vee\tilde X_0]}{\varepsilon t^n}\int_0^t{\e^{-\lambda\rho w}(t-w)w^{n-2}dw}.\\
\end{align*}
Then
\begin{align*}
\prob\left(\left|X_{T_{N_t}^-}-\tilde X_{T_{N_t}^-}\right|\geq\varepsilon\right)&\leq\e^{-\lambda t}(1+\lambda t)+\frac{\E[X_0\vee\tilde X_0]}{\varepsilon}\e^{-\lambda t}\sum_{n=2}^\infty{\int_0^t{\frac{\lambda^n}{(n-2)!}\e^{-\lambda\rho w}(t-w)w^{n-2}dw}}\\
	&\leq\e^{-\lambda t}(1+\lambda t)+\frac{\E[X_0\vee\tilde X_0]}{\varepsilon}\lambda^2\e^{-\lambda t}\int_0^t{\e^{-\lambda\rho w}\e^{\lambda w}(t-w)dw}\\
	&\leq \e^{-\lambda t}\left(1+\lambda t+\frac{\E[X_0\vee\tilde X_0]}{\varepsilon(1-\rho)^2}\left(\e^{\lambda(1-\rho)t}-1-\lambda(1-\rho)t\right)\right).
\end{align*}
Then, we use Proposition~\ref{vitTV}, Lemma~\ref{majSupEta} and \eqref{Eq_majSupEta4} to conclude.
\end{proof}

Now, let us develop the inequality given in Proposition~\ref{vitExpo}:
\begin{align*}
\|\mu_0 P_t-\tilde\mu_0 P_t\|_{TV}\leq &K\frac{M+1}2\varepsilon^h+(1+\lambda t)\e^{-\lambda t}-K\frac{M+1}2(1+\lambda t)\e^{-\lambda t}\varepsilon^h+\frac{\E[X_0\vee\tilde X_0]}{\varepsilon(1-\rho)^2}\e^{-\lambda\rho t}\\
	&-\frac{K(M+1)\E[X_0\vee\tilde X_0]}{2\varepsilon(1-\rho)^2}\e^{-\lambda\rho t}\varepsilon^h-\frac{\E[X_0\vee\tilde X_0]}{\varepsilon(1-\rho)^2}(1+\lambda(1-\rho)t)\e^{-\lambda t}\\
	&\frac{K(M+1)\E[X_0\vee\tilde X_0]}{2\varepsilon(1-\rho)^2}(1+\lambda(1-\rho)t)\e^{-\lambda t}\varepsilon^h
\end{align*}
The only fact that matters is that the first and the fourth terms in the previous expression are the slowest to converge to 0, thus it is straightforward that the rate of convergence is optimized by setting
\[\varepsilon=\exp\left(-\frac{\lambda\rho}{1+h}t\right),\]
and then there exists $C_2>0$ such that
\begin{equation}
\|\mu_0 P_t-\tilde\mu_0 P_t\|_{TV}\leq C_2\exp\left(-\frac{\lambda\rho h}{1+h}t\right).
\label{Eq_VitPartCase2}
\end{equation}

One can easily conclude, by comparing \eqref{Eq_VitPartCase1} and \eqref{Eq_VitPartCase2} that the second method provides a strictly better lower bound for the speed of convergence of the process to equilibrium.

\bibliography{Biblio}

\begin{acknowledgements}
I would like to thank my advisors, Jean-Christophe Breton and Florent Malrieu, for giving many useful advices and comments. This work is part of my Ph.D. thesis at the University of Rennes 1, France, and was supported by the Centre Henri Lebesgue (programme "Investissements d'avenir" --- ANR-11-LABX-0020-01).
\end{acknowledgements}

\end{document}